\documentclass[11pt]{amsart}

\usepackage{amsaddr}

\usepackage{euscript}
\usepackage{amsmath}
\usepackage{amsthm}
\usepackage{epsfig}
\usepackage{amssymb}\usepackage{amscd}
\usepackage{epic}
\usepackage{color}

\usepackage{tikz-cd} 

\usepackage{tikz}
\usetikzlibrary{shapes.geometric, arrows}

\tikzstyle{block} = [rectangle, rounded corners, minimum width=2.5cm, minimum height=1cm, text width=3.5cm, text centered, draw=black]
\tikzstyle{arrow} = [thick,->,>=stealth]

\usepackage{graphicx}   

\numberwithin{equation}{section}
\numberwithin{equation}{subsection}

\theoremstyle{plain}

\newtheorem{lemma}[equation]{Lemma}

\newtheorem{thm}[equation]{Theorem}

\newtheorem{prop}[equation]{Proposition}

\theoremstyle{definition}

\newtheorem{defn}[equation]{Definition}
\newtheorem{nota}[equation]{Notation}

\newtheorem{rem}[equation]{Remark}

\numberwithin{equation}{section}
\numberwithin{equation}{subsection}

\newcommand{ \lk }{ \mbox{lk} }

\newcommand{ \imm }{ \mbox{Imm} }
\newcommand{ \emb }{ \mbox{Emb} }

\def\C{\mathbb C}

\def\R{\mathbb R}
\def\Z{\mathbb Z}

\def\im{{\rm Im}}

\def\d{\mathrm{d}}

\title[Cross-caps, triple points and a linking invariant]{Cross-caps, triple points and a linking invariant for finitely determined germs}

\author[G. Pint\'er]{Gerg\H{o} Pint\'er}
\address{Department of Theoretical Physics, Institute of Physics, Budapest University of Technology and Economics, M\H{u}egyetem rkp. 3., H-1111 Budapest, Hungary}
\email{pinter.gergo@ttk.bme.hu}

\author[A. S\'andor]{Andr\'as S\'andor}
\address{
Alfr\'ed R\'enyi Institute of Mathematics, Re\'altanoda u. 13-15, Budapest, 1053, Hungary  \\
Central European University, Dept. of Mathematics, N\'ador u. 9., Budapest, 1051, Hungary}

\email{sandora@renyi.hu}


\date{}

\begin{document}

\maketitle

\begin{abstract}
	It was recently proved that for finitely determined germs $ \Phi: ( \C^2, 0) \to ( \C^3, 0) $ the number $C(\Phi)$ of Whitney umbrella points and the number $T(\Phi)$ of triple values of a stable deformation are topological invariants. The proof uses the fact that the combination $C(\Phi)-3T(\Phi)$ is topological since it equals the linking invariant of the associated immersion $S^3 \looparrowright S^5$ introduced by Ekholm and Sz\H{u}cs. We provide a new, direct proof for this equality. We also clarify the relation between various definitions of the linking invariant.
\end{abstract}

\section{Introduction}

\subsection{} Let $ \Phi: ( \C^2, 0) \to ( \C^3, 0) $ be a finitely determined (also called $ \mathcal{A}$-finite) holomorphic germ. In this case $ \mathcal{A}$-finiteness means that $ \Phi $ is  a stable immersion off the origin \cite{Wall, mond-ballesteros}. For these germs the number  of the complex Whitney umbrella (cross cap) points $ C( \Phi)$ and the triple values $ T( \Phi) $ of a stable holomorphic deformation are well-defined analytic invariants \cite{Mond0,Mond2}.
Recently in \cite{dBM} J. Fern\'andez de Bobadilla, G. Pe\~nafort, and J. E. Sampaio proved that these invariants are topological, moreover they are determined by the embedded topological type of the image of $ \Phi $. One of the main ingredients of their proof is the formula 
\begin{equation}\label{eq:main}
L( \Phi|_{ \mathfrak{S}})=C( \Phi)-3T( \Phi) 
\end{equation}
from \cite{NP}, which expresses the naturally topological Ekholm--Sz\H{u}cs invariant (also called triple point invariant or linking invariant) $ L( \Phi|_{ \mathfrak{S}})$ of the associated stable immersion
$\Phi|_{\mathfrak{S}}: \mathfrak{S} \simeq S^3 \looparrowright S^5$ in terms of $ C$ and $T $.
However, the formula (\ref{eq:main}) is proved in \cite{NP} in a rather complicated way, by using two Smale invariant formulas. The main purpose of this article is to provide a new direct proof for this formula. 

The Ekholm--Sz\H ucs invariant $L(f) $ of a stable immersion $ f: S^3 \looparrowright \R^5$ measures the linking of the image with a copy of the double values, shifted slightly along a suitable chosen normal vector field. 
 In the literature different versions of the definition of $L$ can be found (see \cite{ekholm3, ekholm4, ESz, saeki}), whose relation is not completely clarified. We verify their equivalence, i.e. $L_1(f)=-L_2(f)$, based on their opposite behavior through regular homotopies.

 Although our proof of the main theorem \eqref{eq:main} is self-contained, an independent secondary goal of this article is to clarify the enigmatic relation between
several variants of the linking invariant $L$ and other related invariants, used in the study of generic $\mathcal{C}^{\infty}$ real maps and immersions. 

\subsection{Structure of the article} In the Preliminaries (Section \ref{se:pre}), we summarise the properties of finitely determined holomorphic germs we will use. We outline the definitions of  $C$ and $T$ and their invariance for analytic, $\mathcal{C}^{\infty}$ and topological left-right equivalence. We introduce the associated immersion and we describe the double point structure of $\Phi$.

In Section \ref{se:link}, we collect the different definitions of the Ekholm--Sz\H{u}cs invariant $L$ of stable immersions $S^3 \looparrowright S^5$ from the literature. We show that they agree up to sign and we clarify that sign. 
Then we define an invariant for finitely determined germs by applying $L$ to the associated immersions, and we prove its topological left-right invariance.

In Section \ref{s:prmain}, we provide a new, direct proof for the correspondence $L=C-3T$. We use local calculations near complex cross cap points and triple values.

Finally, Appendix \ref{se:elojel} is a brief summary of the applications of $L$ and another similar linking invariant in the study of generic real maps and immersion theory. We collect the most relevant results and clear up the context of this article, including the main steps of the original proof of \eqref{eq:main}. Then we compare the new local calculation for the complex cross cap points with an older one in \cite{NP}, and clarify its consequences for the Ekholm--Sz\H{u}cs Smale invariant formula.

\subsection*{Acknowledgements} We are very grateful to 
L\'aszl\'o Feh\'er, 
Andr\'as Sz\H{u}cs,
Tam\'as Terpai,
Andr\'as N\'emethi, 
J\'ozsef Bodn\'ar and 
David Mond
 for
several very helpful conversations regarding this topic. We also thank the reviewer for their insightful suggestions.

GP thanks his physicist colleagues Andr\'as P\'alyi, Gy\"orgy Frank, D\'aniel Varjas and J\'anos Asb\'oth for the new inspiration to the singularity theory research.

\section{Preliminaries}\label{se:pre}

\subsection{Invariants of a stabilization}\label{ss:invstab} A holomorphic germ $ \Phi: ( \C^2, 0) \to ( \C^3, 0) $  is \emph{finitely $ \mathcal{A}$-determined} (briefly, \emph{finitely determined}), if there is an integer $k$ such that the $k$-th Taylor polynomial of $ \Phi $ determines it up to left-right equivalence, or equivalently, the $ \mathcal{A}$-codimension of $ \Phi$ is finite.
By Mather--Gaffney criterion \cite{Wall, mond-ballesteros}, $ \Phi $ is finitely determined if and only if its restriction $ \Phi|_{ \C^2 \setminus \{ 0 \}} $ is stable. This means that a sufficiently small representative of $\Phi|_{ \C^2 \setminus \{ 0 \}}$ has only (1) regular simple points and (2) double values with transverse intersection of the regular branches.

The only possible multigerms of a stabilization (stable deformation) of a holomorphic germ $ \Phi: ( \C^2, 0) \to ( \C^3, 0) $ are (1) regular simple points, (2) double values with transverse intersection of the regular branches, (3) triple values with regular intersection of the regular branches and (4) simple \emph{Whitney umbrella (cross cap)} points. The Whitney umbrellas and the triple values are isolated points, up to analytic $\mathcal{A}$-equivalence they have local normal forms
\begin{equation}\label{eq:wh} 
    \mbox{Whitney umbrella (cross cap): } \ (s,t) \mapsto (s^2, st, t) 
    \end{equation}
    \begin{equation}\label{eq:tr}
    \mbox{Triple value:} 
    \left\{ \begin{array}{ccc} 
     (s_1,t_1)    & \mapsto & (0,s_1,t_1)  \\
     (s_2,t_2)    & \mapsto & (t_2,0,s_2)  \\
     (s_3,t_3)    & \mapsto & (s_3,t_3,0)  \\
    \end{array} \right.
\end{equation}

The numbers $ C( \Phi ) $ of the cross caps and $ T( \Phi ) $ of the triple values are independent of the stabilization, they are analytic invariants of the finitely determined germs $ \Phi $.
Both invariants were  introduced by Mond \cite{Mond0, Mond2}, they can be defined in algebraic way as well, without referring to a stabilization, as follows. 

Let $ C_{alg}( \Phi ) $ be the codimension of the \emph{ramification ideal}, which is the ideal in the local ring $ \mathcal{O}_{( \C^2, 0)}$ generated by the determinants of the $2 \times 2$ minors of the Jacobian matrix of $ \Phi : ( \C^2, 0) \to ( \C^3, 0) $. $ T_{alg}( \Phi ) $ is  the codimension of the second Fitting ideal associated with $ \Phi $ in $ \mathcal{O}_{( \C^3, 0)}$ \cite{mondfitting}. If $ \Phi $ is finitely determined, then both $ C_{alg}( \Phi ) $ and $ T_{alg}( \Phi ) $ are finite, and any stabilization of $ \Phi $ has $ C( \Phi )= C_{alg}( \Phi )$ number of cross caps and $ T( \Phi ) =T_{alg}( \Phi )$ number of triple values.   The invariants $ T$ and $ C$ appear in several different contexts, see for example \cite{Mondvan, mararmulti, nunodouble, slicing, mond-ballesteros, gtezis}.

The analytic invariance of $C$ and $T$ means the following. Let $ \Phi_1$ and $ \Phi_2$ be finitely determined germs, analytic $\mathcal{A}$-equivalent to each other. That is, there exist germs of biholomorphisms $ \phi: (\C^2, 0) \to ( \C^2, 0)$ and $ \psi: (\C^3, 0) \to ( \C^3, 0)$ such that 
\begin{equation}\label{eq:lr} 
\Phi_2=\psi \circ \Phi_1 \circ \phi \end{equation}
holds, i.e. the diagram below commutes.

\begin{equation} \label{diag:lr}
\begin{tikzcd}
( \C^2, 0) \arrow[r, "\Phi_1"] 
& ( \C^3, 0) \arrow[d, "\psi"] \\
( \C^2, 0) \arrow[r, "\Phi_2"]
\arrow[u, "\phi"]
&  ( \C^3, 0)
\end{tikzcd}
\end{equation}

Then \begin{equation}\label{eq:invct}C(\Phi_1)=C(\Phi_2) \mbox{ and } T(\Phi_1)=T(\Phi_2).
\end{equation}

In \cite{NP} it is proved that $C$ and $T$ are $ \mathcal{C}^{\infty}$-invariants as well. That is, (\ref{eq:invct}) holds also for $ \mathcal{C}^{\infty}$ left-right equivalent germs, i.e. for two holomorphic finitely determined germs for which (\ref{diag:lr}) holds with some germs of $\mathcal{C}^{\infty} $-diffeomorphisms  $ \phi: (\R^4, 0) \to ( \R^4, 0)$ and $ \psi: (\R^6, 0) \to ( \R^6, 0)$. (Here, $\C^n$ and $ \R^{2n}$ are naturally identified.)

The topological invariance of $C$ and $T$ would mean that (\ref{eq:invct}) holds also for topologically left-right equivalent germs, that is when we only require $\phi$ and $ \psi$ to be germs of homeomorphisms. This invariance was an open question for  a long time. In \cite{NP} A. N\'emethi and the first author proved that the linear combination $C-3T$ is a topological invariant. This follows from $L=C-3T$ (formula (\ref{eq:main})) which expresses a topological invariant (the Ekholm--Sz\H ucs invariant) of the associated immersion, see the next sections.  In this article, we present a new direct proof of formula $L=C-3T$. 
(We also prove the topological invariance of the Ekholm--Sz\H ucs invariant, see Proposition \ref{prop:topinv}. This fact is very natural and has been implicitly used previously, but according to the authors' knowledge, it has not been published yet.)

In \cite{dBM} J. Fern\'andez de Bobadilla, G. Pe\~nafort, and J. E. Sampaio proved that $C$ and $T$ are topological invariants, moreover they are determined by the embedded topological type of the image of $ \Phi $. A key ingredient of their proof is the topological invariance of $C-3T$, which follows from the formula $L=C-3T$. 


\subsection{The associated immersion} Let $ \Phi: (\C^2, 0) \to (\C^3,0) $ be  a finitely determined germ. Such a germ, on the level of links of the spacegerms
$(\C^2,0)$ and $(\C^3,0)$, provides
a stable immersion $ \Phi|_{S^3}: S^3 \looparrowright S^5 $ as follows. The preimage $ \mathfrak{S}:= \Phi^{-1} ( S^5_{ \epsilon })$ of the $5$-sphere $ S^5_{ \epsilon } \subset \C^3 $ around the origin, with a sufficiently small radius $ \epsilon $,  is diffeomorphic to $S^3$. The restriction 
$ \Phi|_{\mathfrak{S}}: \mathfrak{S} \looparrowright S^5_{ \epsilon} $ is the immersion associated with $ \Phi $. The regular homotopy class of $\Phi|_{\mathfrak{S}} $ is independent of all the choices. The immersions obtained by different choices are regular homotopic to each other through stable immersions. See \cite[2.1.]{NP} or \cite[Subsection 1.1.2.]{gtezis}.

\subsection{The image and the double points} Write $(X,0)$ for $({\rm im}(\Phi),0)$ and
let $ f: ( \C^3, 0) \to ( \C, 0) $ be the reduced equation of $(X,0)$.
 Note that $ (X, 0) $ is a {\it non-isolated} hypersurface singularity, except when $ \Phi $ is a regular map (see \cite{NP}).
 We denote by $ (\Sigma,0) = ( \partial_{x_1} f, \partial_{x_2} f, \partial_{x_3} f)^{-1} (0) \subset (\C^3,0) $
 the {\it reduced} singular locus of  $(X,0)$
  -- that is the closure of the set of double values of $\Phi$.
 Also, we denote by $(D,0)$ the {\it reduced}
 double point curve $ \Phi^{-1} (\Sigma ) \subset (\C^2,0) $. The reduced equation of $D$ is
$ d: (\C^2, 0) \to (\C, 0) $.
 (In fact, the finite determinacy of the germ $ \Phi$
 is equivalent with the fact that the double point curve $D$ is reduced;  see e.g. \cite{nunodouble}.)
 
 Let $ \Upsilon \subset S^5_\epsilon$ be the link of $\Sigma$.
It is exactly the  set of double values  of $  \Phi |_{ \mathfrak{S}} $.
Let $ \gamma = \Phi^{-1}(\Upsilon) \subset \mathfrak{S}^3 $ denote the set of double points of
$  \Phi |_{ \mathfrak{S}} $, that is,
$ \gamma \subset  \mathfrak {S}  $ is the link of $D$. All link components are considered with  their natural orientations.

\begin{equation}
     \begin{array}{ccc}
( \C^2, 0) & \to & ( \C^3, 0) \\
\cup &          & \cup \\
(D, 0) &  \to      & ( \Sigma, 0)  \\
 \cup &          & \cup \\
\gamma=D \cap \mathfrak{S}^3 &  \to      & \Upsilon= \Sigma \cap S_\epsilon^5            
\end{array}
\end{equation}

\begin{figure}[h]

\centering
\includegraphics[width=0.9\textwidth]{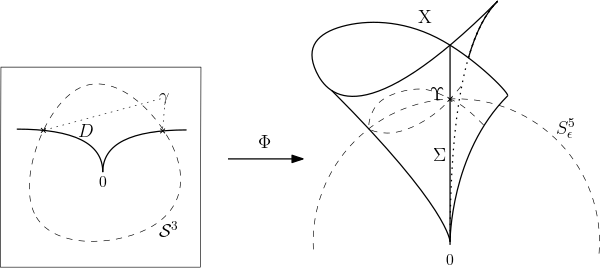}
\caption{Notations of the various parts of the space germs.}
\label{fig:map}
\end{figure}

\section{The Ekholm--Sz\H{u}cs linking invariant}\label{se:link}

\subsection{The Ekholm--Sz\H ucs invariant of stable immersions} \label{ss:eszinv} The invariant $L(f)$ of a stable immersion $ f: S^3 \looparrowright \R^5 $ measures the linking of a shifted copy of the double values with the whole image of $f$. Different versions of the definition can be found in the literature, for references see below.  In this paragraph, we review these definitions and prove their equivalence via their behavior along regular homotopies. We present the whole argument in the simplest case, for immersions $S^3 \looparrowright \R^5$, although originally they were introduced for different levels of generality (for other manifolds, higher dimensions) in \cite{ekholm3, ekholm4, ESz, saeki}. This discussion is an extended version of the summery in \cite[2.2.2.]{gtezis}.

A stable immersion $ f: S^3 \looparrowright \R^5 $ has only simple values and double values with transverse intersection of the two branches. Let 
$ \gamma \subset S^3 $ be the double point locus of $f$,
that is
$ \gamma = \{ p \in S^3 \ | \ \exists p' \in S^3: \ p \neq p' \mbox{ and } f(p)=f(p') \} $.
The locus $ \gamma $ is a closed $1$-manifold, i.e. a link in $S^3$ with possibly more components.
The map $ f|_{\gamma}: \gamma \to f( \gamma ) $ is a $2$-fold covering.
$ \gamma $ is endowed with an involution $ \iota: \gamma \to \gamma $ such that $\iota(p) \neq p$ and $ f( p )= f ( \iota (p) ) $ hold for all $ p \in \gamma $.

The first definition of $ L(f)$ is from \cite[6.2.]{ekholm3}. Let $v$ be a vector field along $\gamma$ tangent to $S^3$ and nowhere tangent to $\gamma$, i.e. $v$ represents a section of the normal bundle $TS^3|_{\gamma}/T\gamma$ of $\gamma \subset S^3$.  We also require that $[ \tilde{\gamma} ]  $ is $0$ in $H_1 ( S^3 \setminus \gamma , \Z) $, where $ \tilde{ \gamma } \subset S^3$ is the result of pushing $ \gamma $ slightly along $v$. Such a vector field $v$ is unique up to homotopy, and for instance each of the two vectors of a Seifert framing provides such a vector field.
If $v$ is such a vector field, then the linking number $ \lk_{S^3}( \gamma, \tilde{ \gamma} )$ equals to $ 0 $, but the reverse is not true, since $\lk_{S^3}( \gamma, \tilde{ \gamma} )$ is the sum of the components of $[ \tilde{\gamma} ]   \in H_1 ( S^3 \setminus \gamma , \Z) $.
(All the linking numbers appearing are considered with respect to the natural orientation of the curves and submanifolds involved.)
Let $ q = f(p) = f ( \iota (p)) $ be a double value of $f$. Then $ w(q) = df_p (v(p)) + df_{ \iota(p)} (v (\iota (p) ) $ defines a vector field $w$ along $ f( \gamma ) $ that is nowhere tangent to the branches of $f$. In this sense $w$ is a normal vector field of $f$ along $f(\gamma)$.
Let $ \widetilde{f(\gamma)} \subset \R^5 $ be the result of pushing $ f( \gamma) $ slightly along $w$, then $ \widetilde{f(\gamma)} $ and $ f( S^3 ) $ are disjoint. The first invariant is the linking number

\begin{equation}\label{eq:L1} 
L_1 (f) := \lk_{ \R^5 } ( \widetilde{f(\gamma)}, f( S^3) )
\end{equation}

(or equivalently, $L_1 (f)= [ \widetilde{f(\gamma)} ] \in 
H_1(\R^5 \setminus f(S^3), \Z) \cong \Z $).
Note that Ekholm used an other notation: in \cite[2.2., 6.2.]{ekholm3} our $ L_1 (f) $ is denoted by $ \lk (f) $, and $ L(f) $ is defined as $ \lfloor \lk (f) /3 \rfloor $.

The second definition is \cite[Definition 11.]{ESz}, \cite[Definition 2.2.]{saeki}. It works only with further assumptions, see Remark \ref{re:fuas} below. The normal bundle $ \nu (f) $ of $f$ is trivial, since the oriented rank--$2$ vector bundles over $S^3 $ are classified by $ \pi_2 (SO(2)) =0$. Any two trivializations are homotopic, since their difference represents an element in $ \pi_3 (SO(2)) =0$. Let $ (v_1, v_2) $ be the homotopically unique normal framing of $f$, and at a double value $ q = f(p) = f ( \iota (p)) $ define 
$ u(q) = v_1(p) + v_1 (\iota (p))  $. $u$ is a normal vector field along $f( \gamma)$, and let $ \overline{f(\gamma)} \subset \R^5 $ be the result of pushing $ f( \gamma) $ slightly along $u$. Then $ \overline{f(\gamma)} $ and $ f( S^3 ) $ are disjoint. The invariant is the linking number (or equivalently, the homology class) 

\begin{equation}\label{eq:L2} 
L_2 (f) := \lk_{ \R^5 } ( \overline{f(\gamma)}, f( S^3) ) = [ \overline{f(\gamma)} ] \in 
H_1(\R^5 \setminus f(S^3), \Z) \cong \Z \mbox{.}
\end{equation}

Note that the framing $(v_1, v_2)$ can be replaced by an arbitrary nonzero normal vector field $v$ of $f$, since it can be extended to a framing whose first component is $v$.

\begin{rem}\label{re:fuas}
	Without further assumptions it is possible that $u(q)$ is tangent to one of the branches of $f$, hence it can happen that $ \overline{f(\gamma)} \cap f( S^3 ) \neq \emptyset $. To avoid this problem one has to choose a unit normal vector field $v$ or has to assume that the intersection of the branches is orthogonal, which can be reached by a regular homotopy through stable immersions. In this paper all the calculations uses $L_1$ and not $L_2$.
	\end{rem}

The third definition is in \cite[Definition 4.]{ESz}, see also \cite[4.5., 4.6.]{ekholm4}. Let $v$ be a nonzero normal vector field of $f$ along $ \gamma $, that is, a nowhere zero section of $ \nu (f) |_{ \gamma} $. Let $ [v] $ be the homology class represented by $v$ in $ H_1 ( E_0( \nu (f)), \Z ) \cong \Z $, where $ E_0( \nu (f)) $ denotes the total space of the bundle of nonzero normal vectors of $f$. Let $u_v (q) = v(p) + v( \iota (p) ) $ be the value of the vector field $u_v$ along $ f( \gamma ) $ at the point $ q= f(p) = f ( \iota (p) )$. Let $ \overline{f(\gamma)}^{(v)} $ be the result of pushing $ f( \gamma) $ slightly along $u_v$, then $ \overline{f(\gamma)}^{(v)} $ and $ f( S^3 ) $ are disjoint. The invariant is
\begin{equation}\label{eq:Lv} 
L_v (f) := \lk_{ \R^5 } ( \overline{f(\gamma)}^{(v)}, f( S^3) ) -[v]= [ \overline{f(\gamma)}^{(v)} ] -[v] \mbox{,}
\end{equation}
where $ [ \overline{f(\gamma)}^{(v)} ] \in 
H_1(\R^5 \setminus f(S^3), \Z) \cong \Z $.

By \cite[Lemma 4.15.]{ekholm4} $ L_v (f) $ is well-defined, that is, $ L_v (f) $ does not depend on the choice of the normal field $v$. Moreover, if $ v $ is the restriction of a (global) normal vector field of $f $ to $ \gamma $, then $[v]=0$. Indeed, the restriction of the normal field of $f$ to a Seifert surface $H$ of $ \gamma $ results a surface $ \overline{H} \subset E_0( \nu (f)) $, 
whose boundary is the image of $ v : \gamma \to E_0( \nu (f))$. Hence $ L_v (f) = L_2 (f) $.

The invariants $L_1 $, $L_2 $ are equal to each other with opposite sign. This follows from the fact that they behave in an inverse way along regular homotopies, i.e. they change with the same number with opposite sign when a stable regular homotopy steps through first order instabilities: immersions with (1) one triple value (``triple point moves'') or (2) a self-tangency (``self-tangency moves''). For definitions we refer  to \cite{ekholm3, ekholm4}. The proof  of Proposition~\ref{pr:Leq} is a result of a discussion with Tam\'as Terpai and Andr\'{a}s Sz\H{u}cs.

\begin{prop}\label{pr:Leq}

	(a) $ L_1 (f) $ and $ L_2 (f) $ are invariants of stable immersions. They change by $ \pm 3 $ under triple point moves and do not change under self tangency moves. In other words: if $f$ and $g$ are regular homotopic stable immersions, $h: S^3 \times [0, 1] \to \R^5 $ is a stable regular homotopy between them, then $\pm (L_i (f) -L_i (g))$ is equal to three times the algebraic number of the triple values of the map $ H: S^3 \times [0, 1] \to \R^5 \times [0, 1]$, $H(x, t)=(h(x, t), t)$.
	
	(b) In the above setup $L_1 (f) -L_1 (g)=-(L_2 (f) -L_2 (g))$.
	
	(c) The three definitions are equivalent:
	\[
	L_1 (f) = -L_2 (f) = -L_v (f).
    \]
\end{prop}

\begin{proof}
	Part (a) is proved for $L_1$ in \cite[Lemma 6.2.1.]{ekholm3} and for $L_2=L_v$ in \cite[Theorem 1.]{ekholm4}. 
	
	For part (b), we compare the change of $L_1$ and $L_2$ through a triple point move. In the proof of \cite[Lemma 6.2.1.]{ekholm3} Ekholm defines a local model of the triple point move where $L_1$ increases by $3$. On the other hand, in  the discussion preceding \cite[Definition 6.3]{ekholm4} he provides a convention to measure the change of $L_2$. If we check this convention on the previous local model, we obtain that $L_2$ decreases by $3$ through that triple point move. Hence $L_1$ and $L_2$ changes in opposite ways at each triple point move.
	
	Using part (a) and part (b), we prove part (c) as follows. 
	Since $ L_1 $ and $L_2 $ changes in opposite way along a regular homotopy, $L_1+L_2$ is a regular homotopy invariant. Moreover
	$ L_1$ and $L_2 $ are additive under connected sum, see \cite[Lemma 5.2., Proposition 5.4.]{ekholm4}, \cite[6.5.]{ekholm3}. It follows that $L_1 + L_2$ defines a homomorphism from $ \imm (S^3, \R^5 ) $ to $ \Z$. If $f: S^3 \hookrightarrow \R^5 $ is an embedding, then $ L_1 (f) = L_2 (f) = 0 $, hence $ L_1 + L_2 $ is $0$ on the $24$-index subgroup $ \emb (S^3, \R^5 )$ of $ \imm ( S^3, \R^5 ) \cong \Z$. It follows that $L_1 + L_2$ is $0$ for every stable immersion, hence $L_1 = - L_2$. 
\end{proof}

We fix the following convention.
\begin{nota}
$L(f):=L_1(f)$.
\end{nota}

In the continuation of this article, $L$ will be  studied  thoroughly, and a new direct proof of $L_2(f)=-L_1(f)$ will be provided.

\subsection{Ekholm--Sz\H{u}cs invariant for finitely determined germs}\label{ss:eszfindet} 

The definition of $L_1(f)$ and $L_2(f)$ of immersions  $f: S^3 \looparrowright \R^5$ cannot be applied directly for $\Phi|_{\mathfrak{S}}: \mathfrak{S} \looparrowright S^5$. In fact, the shifted copy  of $\upsilon \in S^5$ by a normal vector field is  a curve in $ \C^3=\R^6$, but not exatly in $S^5$. To solve this technical difficulty we recall one of the definitions of the linking number.

\begin{defn} \label{def:linking}
Let $N^n, M^m \subset S^k= \partial B^{k+1}$ be two closed oriented submanifolds with dimensions $n+m+1=k$. 
Choose any oriented homological membranes $\widetilde{M}, \widetilde{N} \subset B^{k+1} $ for them, that is, $\widetilde{M}$ and $ \widetilde{N}$
 are singular chains in $ B^{k+1} $ of dimensions $n+1$, respectively $m+1$, with coefficients in $\Z$,  whose boundaries are $\partial \widetilde{N} =N$, $\partial \widetilde{M}=M$. 
 Then the linking number  $ \lk_{S^k} (N, M)$ of $N$ and $M$ in $S^k$ is defined as the intersection number $\mbox{int}_{B^{k+1}} (\widetilde{M}, \widetilde{N})$ of $\widetilde{M}$ and $\widetilde{N} $ in $ B^{k+1} $.
\end{defn}

For the definition of $L_1(\Phi|_{ \mathfrak{S}})$ consider $\overline{\mbox{grad} (d)}$, the conjugate of the gradient vector field of $d$ defined on $D$. Its restriction to $\gamma \subset \mathfrak{S}$ is a representative of the homotopically unique Seifert framing of $\gamma$. Then the sum of the two copies of $d \Phi ( \overline{\mbox{grad} (d)})$ is a nonzero normal vector field along $\Sigma \setminus \{0 \}$, which extends to the origin with $0$. Let $\widetilde{ \Sigma}$ be a copy of $\Sigma $ shifted along this vector field. Define $\widetilde{\Upsilon}:= \widetilde{\Sigma} \cap S^5$ and $L_1( \Phi|_{\mathfrak{S}})=\lk_{S^5} (\widetilde{\Upsilon}, \Phi(\mathfrak{S}))$. The invariant $L_1( \Phi|_{\mathfrak{S}})$ is equal to the intersection number of any pair of membranes in $B^6$ with boundaries $ \widetilde{\Upsilon}$ and $ \Phi|_{\mathfrak{S}}$. Especially, $L_1( \Phi|_{\mathfrak{S}})$ is the intersection number of $\widetilde{\Sigma}$ and $X$. Unfortunately, however, they intersect each other only at the origin, which is a singular point of possibly both membranes, hence the intersection number cannot be calculated directly. Instead, we will repeat the whole procedure with the \emph{analytic stabilization} of $\Phi$, and that will lead to the formula $L_1( \Phi|_{\mathfrak{S}})=C(\Phi)-3T(\Phi)$. 

$L_2(\Phi|_{ \mathfrak{S}})$ can be defined in a similar way, by using $\overline{\partial_s \Phi \times \partial_t \Phi}$ as a representative of the homotopically unique global normal field of $\Phi|_{\mathfrak{S}}$. We can define the shifted copy $\widetilde{\Sigma}^{(2)}$ of $\Sigma$, and $\widetilde{\Upsilon}^{(2)}:= \widetilde{\Sigma}^{(2)} \cap S^5$. However, by Remark \ref{re:fuas}, we cannot guarantee that $\widetilde{\Upsilon}^{(2)}$ and $\Phi (\mathfrak{S})$ are disjoint. Although the formula $L_2( \Phi|_{\mathfrak{S}})=3T(\Phi)-C(\Phi)$ can be supported by local calculation, the precise proof in this way is technically complicated. On the other hand, $L_2$ can be computed directly for the Whitney umbrella to support that $L_1=-L_2$ holds, see Appendix~\ref{se:elojel}.

The topological invariance of $L(\Phi|_{\mathfrak{S}})=L_1(\Phi|_{\mathfrak{S}})$ is almost trivial, since the linking number is  a topological (homological) invariant. However, its proof has been nowhere explained in detail.

\begin{prop} \label{prop:topinv} $L(\Phi|_{\mathfrak{S}})$ is a topological invariant of $\Phi$. That is if $ \Phi_1$ and $ \Phi_2$ are finitely determined germs topologically $\mathcal{A}$-equivalent to each other (see Subsection \ref{ss:invstab}), then
\begin{equation}\label{eq:inv} L(\Phi_1|_{\mathfrak{S_1}})=L(\Phi_2|_{\mathfrak{S_2}}).
\end{equation}
\end{prop}

\begin{proof}
The topological equivalence of the germs means that there exist germs of  homeomorphisms $ \phi: (\C^2, 0) \to ( \C^2, 0)$ and $ \psi: (\C^3, 0) \to ( \C^3, 0)$ such that 
$\Phi_2=\psi \circ \Phi_1 \circ \phi$
holds. The double point curves $(D_1, 0)=(d_1^{-1}(0), 0)$ of $\Phi_1$ and $(D_2, 0)=(d_2^{-1}(0), 0)$ of $\Phi_2$ are topologically equivalent germs of curves, in fact, $D_1=\phi(D_2)$. 
Their links $\gamma_1$, $\gamma_2$ are of the same type as links in $\mathfrak{S}_1\cong \mathfrak{S}_2\cong  S^3$. 

Although the normal vector field $\mbox{grad} (d_2)$ along $\gamma_2$ cannot be pushed forward by $\phi$ since it is not necessarily differentiable, the slightly pushed out copy $\widetilde{\gamma_2}$ can be. 
The image $\phi(\widetilde{\gamma_2})$ determines a normal vector field denoted by $\phi_* (\mbox{grad} (d_2))$ along $\gamma_1$, which is homotopic to $\mbox{grad}(d_1)$ since both vector field represent the Seifert framing.
Hence the sum of the two copies of $d \Phi_1 (\mbox{grad}(d_1))$ and $d \Phi_1 (\phi_* (\mbox{grad} (d_2)))$ are homotopic normal fields along $\Upsilon_1$, thus the pushed out copies $\widetilde{\Upsilon}_1$ and $\widetilde{\Upsilon}_1^{(2)}$ of $\Upsilon_1$ along these vector fields are homotopic in $S^5 \setminus \Phi_1(\mathfrak{S}_1)=S^5 \setminus \Phi_1(\phi(\mathfrak{S}_2))$. 
Therefore, $\lk_{S^5} (\Phi_1(\mathfrak{S}_1), \widetilde{\Upsilon}_1)=\lk_{S^5} (\Phi_1(\mathfrak{S}_1), \widetilde{\Upsilon}_1^{(2)})$. Finally, applying $\psi$ to the whole configuration does not change the linking numbers, and $\psi (\Phi_1(\mathfrak{S}_1))=\Phi_2( \mathfrak{S}_2)$, $\psi(\Upsilon_1)=\Upsilon_2$, $\psi(\widetilde{\Upsilon}_1^{(2)})=\widetilde{\Upsilon}_2$.
\end{proof}

\begin{rem}
$L(f)$ can be defined for stable immersions $f: M^3 \looparrowright \R^5$ of closed oriented $3$-manifolds $M^3$, with trivial normal bundle, see \cite[Definition 2.5.]{saeki}. Especially $M^3$ can be a disjoint union of some copies of $S^3$. In this way for multigerms $\Phi=(\Phi_i): \sqcup (\C^2, 0)_i \to (\C^3, 0)$ the invariant $L(\Phi|_{M^3})$ is defined, where $M^3= \sqcup \mathfrak{S}_i$ with $\mathfrak{S}_i=\Phi_i^{-1}(S^5_{\epsilon})$. We will use this extension of $L$ for ordinary triple values.
\end{rem}

\begin{rem}
Recall Remark 2.2.7 from \cite{gtezis}. $L$ can be defined also for nonstable immersions which do not have triple
values, by the following argument. Any immersion $f$ admits a small perturbation by regular homotopy to a stable
immersion $\widetilde{f}$, and if $f$ does not have triple values, then any two stable perturbations can
be joined with a regular homotopy without stepping through a triple point. Thus $L(f)$ can
be defined as $L(
\widetilde{f})$ of any small stable perturbation $\widetilde{f}$ of $f$.

Consequently $L(\Phi|_{\mathfrak{S}})$ can be defined not only for finitely determined germs but for germs with finite $C$ and $T$, since for these germs $\Phi|_{\mathfrak{S}}$ is not a stable immersion, but it does not have triple points.  Moreover the equation \eqref{eq:main} holds for these germs too, since the proof uses an analytic stabilization of $\Phi$, not $\Phi$ itself. See also Corollary 3.6.3., Remark 3.6.4. in  \cite{gtezis} or \cite{NP}. An interesting example is $\Phi(s, t)=(s^2, t^2, st)$, which is the double cover of the $A_1$ singularity. See Subsection 3.7.2. of \cite{gtezis}.

However it is not clear for these germs, how can $L(\Phi|_{\mathfrak{S}})$ be computed directly from the topology of $\Phi$, without stabilizing it.

\end{rem}

\section{Main theorem}\label{s:prmain}

\subsection{Proof of the main theorem}
\begin{thm} \label{thm:c-3t}
For a finitely determined holomorphic germ $\Phi: ( \C^2, 0) \to ( \C^3, 0)$   $$L( \Phi|_{ \mathfrak{S}})=C( \Phi)-3T( \Phi). $$
\end{thm}

\begin{proof}
Let $ \Phi_\lambda: \mathfrak{B}_{\lambda} \to B^6_{\epsilon} $ be an analytic stabilization of 
$ \Phi_0=\Phi $.
Here, $ \mathfrak{B}_{\lambda} = \Phi_{\lambda}^{-1} (B^6_{\epsilon}) $ 
with boundary $ \partial \mathfrak{B}_\lambda = \mathfrak{S}_\lambda = \Phi_\lambda^{-1} (S^5_\epsilon) $.

Decreasing $ \lambda $ to $0$ induces a diffeomorphism $ \mathfrak{B}_{\lambda} \simeq \mathfrak{B}$, respectively $\mathfrak{S}_{\lambda} \simeq \mathfrak{S}$, and a regular homotopy through stable immersions between $ \Phi_{\lambda}|_{\mathfrak{S}_{\lambda}} $ and $ \Phi|_{\mathfrak{S}} $. It implies -- as recognised in \cite[Section 9]{NP} -- that
\begin{equation} \label{eq:stab}
L(\Phi|_{\mathfrak{S}}) = L( \Phi_{\lambda}|_{\mathfrak{S}_{\lambda}}).
\end{equation}
We denote the corresponding double point sets respectively by $D_\lambda$, $ \Sigma_\lambda $, $\gamma_\lambda$, $ \Upsilon_\lambda $, as defined in Subsection \ref{ss:eszfindet} . The reduced equation of $D_{\lambda}$ is $d_{\lambda}: \mathfrak{B}_{\lambda} \to \mathbb{C}$.

We are going to count $L(\Phi_\lambda)=L_1(\Phi_\lambda)$. According to definitions \ref{eq:L1} and \ref{def:linking}, we want to construct membranes bounding $\Phi_\lambda(\mathfrak{S}_\lambda)$ and $\widetilde{\Upsilon}_\lambda$, and count their intersection number. The first membrane is simply the whole image $\Phi_\lambda (\mathfrak{B}_\lambda)$.

For the second membrane, consider the normal vector field $w=\overline{\mbox{grad}(d_{\lambda})}$ on $D_\lambda \subset \mathfrak{B_\lambda}$, its restriction represents the Seifert framing on $\gamma_\lambda \subset \mathfrak{S}_\lambda$.
The pushforward $\d \Phi_\lambda (w)$ gives a double valued vector field at each point of $\Sigma_\lambda$. We add up the two vectors pointwise and pushout $\Sigma_\lambda$ slightly along the obtained vector field $v$ to get $\widetilde{\Sigma}_\lambda$.
(Notice that at triple points the vector field $v$ has three values, but they are all zeroes.)

By the  construction in Subsection \ref{ss:eszfindet}, the boundary is $ \partial \widetilde{\Sigma}_\lambda= \widetilde{\Upsilon}_\lambda$ and
\begin{equation} \label{eq:int}
  L( \Phi_\lambda|_{\mathfrak{S}_{\lambda}}) = \mathrm{int} (\Phi_\lambda(\mathfrak{B}_\lambda), \widetilde{\Sigma}_\lambda).
\end{equation}

As the two components of $v$ are tangent to the two branches of the image at a double point, the pushout $\widetilde{\Sigma}_\lambda$ has no intersection point with the whole image near an ordinary double point.

Besides double points, the only two types of singular points that may occur in the stabilized map $\Phi_\lambda$ are Whitney umbrella points and triple points. With the above remark, it means that we only have to count the intersection points near these points.

Umbrella points and triple points are  left-right equivalent to the standard copies of them, see (\ref{eq:wh}) and (\ref{eq:tr}). In the following two lemmas, we calculate the intersection numbers for these normal forms -- which are, in fact, the Ekholm--Sz\H{u}cs invariants of these (multi)-germs. After stating the lemmas we will deduce the global invariant by gluing these pieces together to complete the proof.

\begin{lemma} \label{le:whl1}
The Ekholm--Sz\H{u}cs invariant of the standard Whitney umbrella
  $$ \Phi (s,t)=(s^2,st,t)$$
is
  $$ L(\Phi|_\mathfrak{S})=1$$
where $\mathfrak{S}=\Phi^{-1}(S^5_\epsilon)$.
\end{lemma}

\begin{lemma} \label{le:trl1}
The standard triple value is the regular intersection of three branches. We parametrize it the following way
  $$ \Phi
    \left\{
    \begin{array}{ccc}
     \Phi_1 : (s_1,t_1)    & \mapsto & (0,s_1,t_1)  \\
     \Phi_2 : (s_2,t_2)    & \mapsto & (t_2,0,s_2)  \\
     \Phi_3 : (s_3,t_3)    & \mapsto & (s_3,t_3,0)  \\
    \end{array}
    \right.$$
(The pairs $(s_i,t_i)$ are local coordinates around the three preimages of the triple point.)
The Ekholm--Sz\H{u}cs invariant of this multigerm is
  $$ L(\Phi|_\mathfrak{S})=-3.$$
\end{lemma}

The above results suggest that each umbrella point and triple point contributes $1$ and respectively $-3$ to the global Ekholm--Sz\H{u}cs invariant. This is, in fact, the case and the brief argument is the following. The (multi)germs at the umbrella and triple points of $\Phi_{\lambda}$ are left-right equivalent of their standard form, hence, by the left-right invariance of $L$ (see \ref{prop:topinv}), the membrane of $\Phi_\lambda$ shall be replaced locally by the one coming from the standard forms.

More precisely, let us take an umbrella point or a triple value $p_i$ in $\C^3$ and take a small balls $U_i \subset \mathbb{C}^3$ around $p_i$ and $V_i \subset \mathbb{C}^2$ around $\Phi_{\lambda}^{-1}(p_i)$ and biholomorphisms $\phi_i: (U_i,p_i) \to (\mathbb{C}^3,0)$ and $\psi_i: (\mathbb{C}^2_r, \underline{0}_r) \to (V_i, \Phi_{\lambda}^{-1}(p_i))$ so that 
\[\phi_i \circ \Phi_\lambda \circ \psi_i: (\mathbb{C}^2_r, \underline{0}_r) \to (\mathbb{C}^3,0) \]
is a standard umbrella (respectively triple point) at $\underline{0}_r$. (Here we use the notation for multi germs: $(\mathbb{C}^2_r, \underline{0}_r)= \bigsqcup_{i=1}^r (\mathbb{C}^2, 0)$ with $r=1$ for a Whitney umbrella point and $r=3$ for a triple value $p_i$.)

We pull back the two membranes of the standard Whitney umbrella (resp. triple value) via $\phi_i$ to define new membranes inside $U_i$. On one hand we obtain another pushout of $\Sigma_{\lambda}$ instead of $\widetilde{\Sigma}_{\lambda}$,
let us denote it by $M_i \subset U_i$. On the other hand, we get back a piece of the other original membrane, $\Phi_{\lambda} (\mathfrak{B}_{\lambda}) \cap U_i$.

Taking a look at the boundary of $U_i$, we find that both $\widetilde{\Sigma}_\lambda \cap \partial U_i$ and $M_i \cap \partial U_i$ have the same linking number with $\Phi_\lambda (\mathfrak{B}_\lambda) \cap \partial U_i$: that is the $L_1$ invariant of the umbrella point or the triple point. Therefore we can construct a collar $N_i$ that connects $\widetilde{\Sigma}_\lambda \cap \partial U_i$ and $M_i \cap \partial U_i$ in $\partial U_i$, in a way that $N_i$ has an intersection number 0 with $\Phi_\lambda (\mathfrak{B}_\lambda) \cap \partial U_i$.

Gluing all these pieces together, we obtain a membrane replacing $\widetilde{\Sigma}_\lambda$:
\begin{equation}
\mathcal{M}=(\widetilde{\Sigma}_{\lambda} \setminus \bigcup_i U_i ) \cup \bigcup_i (N_i \cup M_i).
\end{equation}

The intersection number
$\mathrm{int} (\Phi_\lambda(\mathfrak{B}_\lambda), M_i)$ equals $1$ for an umbrella point and $-3$ for a triple value,
$\mathrm{int} (\Phi_{\lambda}(\mathfrak{B}_\lambda),N_i)=0$ and $\mathrm{int} (\Phi_{\lambda}(\mathfrak{B}_\lambda),(\widetilde{\Sigma}_{\lambda} \setminus \bigcup_i U_i ))=0$, hence

\begin{equation}
L( \Phi|_{ \mathfrak{S}})= \mathrm{int} (\Phi_{\lambda}(\mathfrak{B}_\lambda), \mathcal{M})=
C( \Phi)-3T( \Phi).
\end{equation}

\end{proof}

\subsection{Proofs of the lemmas}

\begin{proof}[Proof of Lemma \ref{le:whl1}] Consider the standard Whitney umbrella $\Phi (s,t)=(s^2,st,t)$.
The closure of the set of double values of $\Phi$ is
  $$\Sigma=\{y=z=0\}=\{ (x,0,0): x\in \C \}.$$
This is the image of the double point curve $D=\{t=0\}=\{(s,0): s\in \C\}$. The link of $D$ is $\gamma$ and $\Phi (\gamma)=\Upsilon$.
We compute the linking number
$\lk_{ S^5 } ( \widetilde{\Upsilon}, \Phi(\mathfrak{S}) )$
by defining membranes bounded by $\widetilde{\Upsilon}$ and  $\Phi(\mathfrak{S})$ and taking their intersection multiplicity.

Let the membrane of $\widetilde{\Upsilon}$ be the shifted copy of the curve of double values $\widetilde{\Sigma}$. More precisely, we push $\Sigma$ out from $X=\mbox{im}(\Phi)$ along the pushforward $\d\Phi(v)$ of the vector field $v(s,0)= \overline{\mbox{grad}(t)}(s, 0)=(0,1)$ that is normal to $D$. The differential of our germ is
  $$\d\Phi (s,t)= \begin{pmatrix} 2s&0\\t&s\\0&1 \end{pmatrix}$$
making the pushforward of the normal vector field
  $$\d\Phi (v(s,0)) = \begin{pmatrix} 2s&0\\0&s\\0&1 \end{pmatrix} \cdot
  \begin{pmatrix}
  0 \\ 1
  \end{pmatrix}=
  \begin{pmatrix}
  0 \\ s \\ 1
  \end{pmatrix}. $$

At any double point $(x,0,0)\in \Sigma$, we have two preimages
\[\Phi^{-1} \{(x,0,0)\}=\{ (\sqrt{x},0),(-\sqrt{x},0) \}. \] 
The pushforward of the normal vectors at these points are $\d \Phi (v(\pm \sqrt{x},0))=(0,\pm \sqrt{x},1)$, hence the sum of the two vectors provides the vector field 
$$ w(x, 0, 0) = 
  \begin{pmatrix}
  0 \\ \sqrt{x} \\ 1
  \end{pmatrix} 
  +
    \begin{pmatrix}
  0 \\ -\sqrt{x} \\ 1
  \end{pmatrix}=
    \begin{pmatrix}
  0 \\ 0 \\ 2
  \end{pmatrix}
  $$
along $\Sigma\setminus \{0 \}$. The vector field $w$ can be extended continuously to the origin as it is constant.
Therefore, when we push the double point out by $w$, we obtain
$(x,0,0)+\delta w(x, 0, 0)= (x,0,2\delta) $ for a some $\delta \ll \epsilon $.

Thus the resulting membrane is
 $$\widetilde{\Sigma}=\{(x,0,2 \delta): x\in \C\} \cap B_\varepsilon.$$

On the other hand, let the membrane of $\Phi(\mathfrak{S})$ be simply the image of the ball $\Phi(\mathfrak{B})=X \cap B_{\epsilon}$. That is
  $$\{(x,y,z): xz^2=y^2 \} \cap B_\varepsilon.$$
  
The two membranes $\widetilde{\Sigma}$ and $X \cap \mathfrak{B}$  intersect transversely at $(0,0,2\delta)$. The sign of the intersection is positive as the two membranes have the complex orientations.
\end{proof}

\begin{figure}[h]

\centering
\includegraphics[width=0.9\textwidth]{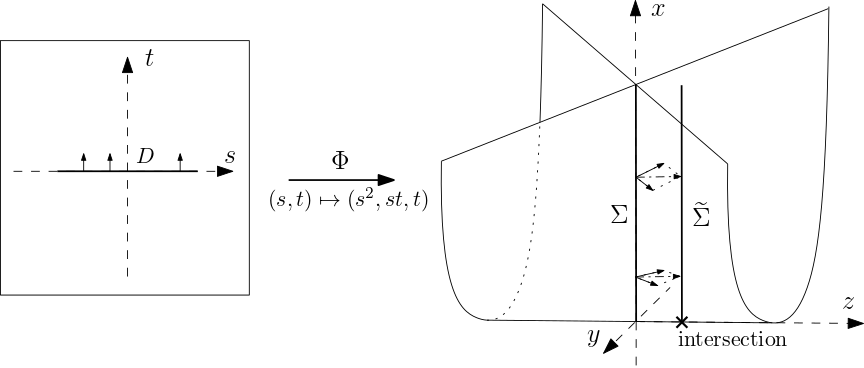}
\caption{Pushing out $\Sigma$ with the sum of the two pushforward vector fields.}
\label{fig:whitney}
\end{figure}

\begin{proof}[Proof of Lemma \ref{le:trl1}] Consider the standard triple value
 $$ \Phi
    \left\{
    \begin{array}{ccc}
     \Phi_1 : (s_1,t_1)    & \mapsto & (0,s_1,t_1)  \\
     \Phi_2 : (s_2,t_2)    & \mapsto & (t_2,0,s_2)  \\
     \Phi_3 : (s_3,t_3)    & \mapsto & (s_3,t_3,0)  \\
    \end{array}
    \right.$$
In this case, the set of double values is
  $$\Sigma=\{(x,0,0)\} \cup \{(0,y,0)\} \cup \{(0,0,z)\}$$
with $x,y,z \in \C$. The curve $\Sigma$ has three components meeting at the origin.
Also, $\Sigma$ has preimages in each two-dimensional chart:
  $$D_i = \{(s_i,0)\} \cup \{(0,t_i)\} = \{s_it_i=0\}$$
for $i\in\{1,2,3\}$.

The membrane we pull over $X \cap S^5$ is again the whole of the image $X\cap B^6$. Note that $X \cap S^5$ is diffeomorphic to the disjoint union of three copies of $S^3$. Thus the membrane consists of three components $X_x=\{x=0\} \cap B, X_y=\{y=0\} \cap B$, and  $X_z=\{z=0\} \cap B$, meeting at the origin.

Now, we describe the membrane for $\widetilde{\Upsilon}$. Let us see what happens if we push out the double values using the sum of the normal vector fields in the preimage -- as before.

The normal vector fields corresponding to $D_i=\{s_it_i=0\}$ are $v_i(s_i,t_i)=\overline{\mbox{grad}(s_it_i)}=(\overline{t_i},\overline{s_i})$.
The differentials of the three map germs are
  $$\d \Phi_1 (s_1,t_1)= \begin{pmatrix} 0 & 0 \\ 1&0 \\ 0&1 \end{pmatrix}, \quad
  \d \Phi_2 (s_2,t_2)= \begin{pmatrix} 0&1 \\ 0&0 \\ 1&0 \end{pmatrix}, \quad
  \d \Phi_3 (s_3,t_3)= \begin{pmatrix} 1 & 0 \\ 0&1 \\ 0&0 \end{pmatrix}.$$
Let us show how our construction works on one component of $\Sigma$. We denote $X_y \cap X_z=\{ (x,0,0): x\in \C \}$ by $\Sigma_x$. A point of $\Sigma_x$ has, again, two preimages: $\Phi_2^{-1}(x,0,0)=(0,x)\in \C_{s_2,t_2}$ and $\Phi_3^{-1}(x,0,0)=(x,0)\in \C_{s_3,t_3}$.
The corresponding normal vectors are $v_2(0,x)=(\overline{x},0)$ and $v_3(x,0)=(0,\overline{x})$.
When we push these vectors forward with the respective differentials, we obtain
  $$\d \Phi_2 (v_2(0,x))= \begin{pmatrix} 0&1 \\ 0&0 \\ 1&0 \end{pmatrix} \cdot
  \begin{pmatrix}
  \overline{x} \\ 0
  \end{pmatrix}
  =
  \begin{pmatrix}
  0 \\ 0 \\ \overline{x}
  \end{pmatrix}
  $$
and similarly $\d \Phi_3 (v_3(0,x))=(0,\overline{x},0)^T$. Hence, by pushing the initial point $(x,0,0)$ out with the sum of these, we reach
  $(x,0,0)+ \delta (0,0,\overline{x})+\delta (0,\overline{x},0)=(x,\delta\overline{x},\delta\overline{x}) \in \widetilde{\Sigma}_x$.
  
Because of the cyclic symmetry of the presentation, the other two components behave similarly, resulting in the membrane
  $$\widetilde{\Sigma}= \{(x,\delta \overline{x}, \delta \overline{x})\} \cup \{(\delta \overline{y},y,\delta \overline{y})\} \cup \{(\delta \overline{z},\delta \overline{z},z)\} = : \widetilde{\Sigma}_x \cup \widetilde{\Sigma}_y \cup \widetilde{\Sigma}_z$$
for some $x,y,z \in \C$ with $\widetilde{\Sigma}$ being in $B$.
One problem with this membrane is that each vector field vanishes at the origin hence in the end we have not moved the point of $\Sigma$ at the origin. Thus $\widetilde{\Sigma}$ meets $X$ only at the origin but with some multiplicity that is somewhat difficult to count.
Fortunately, each pair of components $(X_\alpha,\widetilde{\Sigma}_\beta)$ intersect transversally. We only need to compute the sign of each such intersection and sum them up.

Take $\widetilde{\Sigma}_x=\{(x,\delta \overline{x}, \delta \overline{x})\}$ first. It intersects $X_x=\{x=0\}$ with positive sign, and the other two with negative -- as the corresponding coordinate functions are antiholomorphic. The membranes $\widetilde{\Sigma}_y$ and $\widetilde{\Sigma}_z$ behave similarly. We can summerize this in the formula
  $$\mathrm{int}_0 \big( \widetilde{\Sigma}_\alpha, X_\beta \big)=
  \left\{ \begin{array}{ll}
        +1 & \text{if } \alpha=\beta \\
        -1 & \text{if } \alpha\neq \beta.
  \end{array} \right.$$

Therefore the total intersection number is
  $$\mathrm{int}_0 \big( \widetilde{\Sigma}, X\big)= \sum_{\alpha,\beta \in \{x,y,z\}} \mathrm{int}_0 \big( \widetilde{\Sigma}_\alpha, X_\beta \big)=3\cdot 1 + 6 \cdot (-1)=-3.$$
\end{proof}

\begin{rem}
Note that we could also move the components of $\widetilde{\Sigma}$ away from the origin in order to see the nine points of intersection apart. A perturbation of the form
$$\Sigma'= \{(x-\varepsilon_1,\delta (\overline{x-\varepsilon_2}), \delta (\overline{x-\varepsilon_3}))\} \cup ...$$
with $|\varepsilon_i| \ll \varepsilon$ would do so. In turn, these modifications would not change the topology of the membrane on the boundary of $B$.
\end{rem}

\section{Final remark about future plans} This article has a continuation in progress, in which the relation of $L(f)$  with the surgery coefficients of the Milnor fiber boundary of $(\im (\Phi), 0) \subset (\C^3, 0)$ -- described in \cite{NP2} -- will be explained. In that paper we will also provide a new direct proof for part (c) of Proposition \ref{pr:Leq}.

\appendix

\section{Outline of related results}\label{se:elojel}

The aim of this section is to clarify the role of $L$  and another related linking invariant in the study of generic $\mathcal{C}^{\infty}$ real maps and immersions, and clear up the context of our result. We also clear up some sensitive sign ambiguities related to the Ekholm--Sz\H{u}cs Smale invariant formula.

The other linking invariant $l$ is defined for real generic maps. While $L$ measures the linking of the double values of an immersion with the image of it, $l$ measures the linking of the set of singular points in the target of a  generic map with the image of the map.

The Ekholm--Sz\H{u}cs formula for the Smale invariant of an immersion uses both linking invariant, $L$ of the immersion and $l$ of a singular Seifert surface of the immersion. The original proof of our main formula \eqref{eq:main} is based on the Ekholm--Sz\H{u}cs Smale invariant formula and the `holomorphic Smale invariant formula' of N\'{e}methi and the first author.

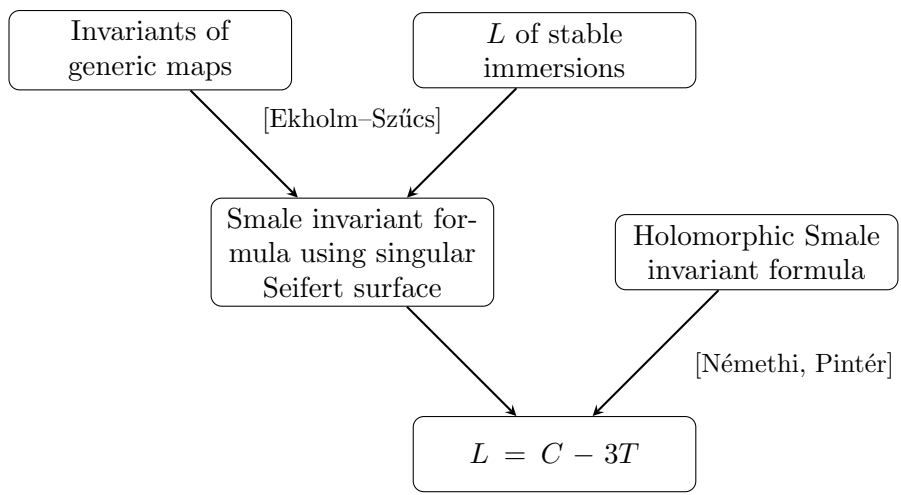
\begin{figure}[h]
 
\centering

\begin{tikzpicture}[node distance=3.8cm]
\node (gen) [block] {Invariants of generic maps};
\node (smale) [block, below right of=gen] {Smale invariant formula using singular Seifert surface};
\node (L) [block, above right of=smale] {$L$ of stable immersions};
\node (eq) [block, below right of=smale] {$L=C-3T$};
\node (holo) [block, above right of=eq] {Holomorphic Smale invariant formula};

\node (ESz) [above of=smale, yshift=-2.05cm] {\small $[$Ekholm--Sz\H{u}cs$]$};
\node (NP) [above right of=eq, yshift=-1.5cm, xshift=0.5cm] {\small $[$N\'emethi, Pint\'er$]$};

\draw [arrow] (gen) -- (smale);
\draw [arrow] (L) -- (smale);
\draw [arrow] (smale) -- (eq);
\draw [arrow] (holo) -- (eq);
\end{tikzpicture}

\caption{Mindmap for the original proof of \eqref{eq:main}.}
\label{fig:gondterkep}
\end{figure}

\subsection{A linking invariant of real generic maps} The $\mathbb{Z}_2$ or integer valued invariant $l(f)$ is defined for real generic maps $f: M^{2k} \to \mathbb{R}^{3k}$ of  closed smooth manifolds $M^{2k}$. It measures the linking of a pushout copy of the singular values with the image of the map as follows. (See, for reference, \cite{trip, ESz, saeki}.)

Such a map $f$ has  (1) regular simple points, (2) double values with transverse
intersection of the regular branches, (3) triple values with regular intersection of the regular
branches and (4) singular values. The dimension of the set of double values is $k$, and the triple values are isolated.
The set of singular values is a $k-1$ dimensional family of generalized real Whitney umbrella points, whose local form is
\begin{equation}
    f_{\mbox{wh}}: (\mathbb{R} \times \mathbb{R}^{k}, 0) \to (\mathbb{R} \times \mathbb{R}^{k} \times \mathbb{R}^{k}, 0),  
\end{equation}
\begin{equation}
    f_{\mbox{wh}}(s, \underline{t})=(s^2, s\underline{t}, \underline{t}).
\end{equation}

The closure $\Delta (f)$ of the set of double values of $f$ is an immersed manifold with boundary. $ \Delta (f)$ has triple self intersection at the triple values of $f$ and the boundary of $\Delta (f)$ is the set of the Whitney umbrella points (singular values) $\Sigma(f) =\partial \Delta (f)$.

The invariant $l(f)$ is defined as the linking number
\begin{equation}
    l(f)=\lk_{\mathbb{R}^{3k}}( \Sigma'(f), f(M^{3k}))
\end{equation}
of the copy $\Sigma'(f)$ of $\Sigma(f)$ shifted slightly along the outward normal field of $\Sigma(f) \subset \Delta (f)$ and the image $f(M^{3k})$ of $f$.  

In general $l(f)$ and the number of triple values is defined only modulo 2 -- because the lack of orientation  on $\Delta (f)$ -- and these $\mathbb{Z}_2$ versions are denoted by $l_2(f)$ and $t_2(f)$ respectively. If $k$ is even and $M^{2k}$ is oriented, then $l(f)$ is a well defined as an integer, \cite{trip}. In these cases each triple value can be given a sign, and the sum of these signs is the integer $t(f)$. 

Ekholm and Sz\H{u}cs expressed some characteristic numbers of $M^{2k}$ in terms of $l$ and $t$. Namely, in \cite{trip} they proved the equality
\begin{equation}\label{eq:z2ek}
    l_2(f)+t_2(f)=\overline{w}^2_k[M] + \overline{w}_{k-1} \overline{w}_{k+1} [M]
\end{equation}
in $\mathbb{Z}_2$, where the terms on the right hand side are products of the normal Stiefel-Whitney classes of $M^{2k}$ evaluated on the fundamental class $[M]$ of $M^{2k}$. 

For $k=2n$ and $M^{2k}$ oriented, the equation of integers
\begin{equation}\label{eq:intek}
    3t(f) -3l(f)=\overline{p}_n [M]
\end{equation}
is proved in \cite{ESz}, where $\overline{p}_n [M]$ is the $n$-th normal Pontryagin number of $M^{4n}$. 
By using Hirzebruch signature theorem, for $k=2$ one can rewrite the formula \eqref{eq:intek} as
\begin{equation}\label{eq:intek4}
    l(f)-t(f)= \sigma(M^4),
\end{equation}
where $ \sigma(M^4)$ is the signature of $M^4$, see \cite{ESz}.

The proofs of these formulas use methods similar to that of our proof. Namely, each of them considers a set of certain type singularities of a map, and deals with the pushout copy of it along a suitably defined normal vector field, then counts the intersection point, see for example \cite[Lemma 3]{ESz}. 
When proving the formula \eqref{eq:z2ek} in \cite[Theorem 1]{trip}, the set of double values of the map $f: M^{2k} \to \mathbb{R}^{3k}$ is shifted slightly along a vector field, which is defined as the sum of the two vectors coming from a suitable normal vector field of the double point set in the source. At this rate it is even more similar to the method we use to prove equation~\eqref{eq:main}.

Despite the similarity in methods, none of the equations \eqref{eq:z2ek}, \eqref{eq:intek} and \eqref{eq:intek4} can be directly applied  for the setup of this article, that is for holomorphic stabilizations $\Phi_{\lambda}$ of holomorphic germs $\Phi: (\mathbb{C}^2, 0) \to (\mathbb{C}^3, 0)$, for the following reasons. 
First, the domain of $\Phi_{\lambda}$ is a 4-ball, which is not a closed manifold, moreover it is topologically trivial. 
Second, the stabilization $\Phi_{\lambda}$ is stable as a holomorphic map, but it is not stable (not generic) in the $\mathcal{C}^{\infty}$-sense, considered as a map from $\mathbb{R}^4$ to $\mathbb{R}^6$. 
Indeed, each isolated complex cross cap point can be further deformed to obtain a stable real $\mathcal{C}^{\infty}$ map with a circle of generalized real cross cap points, see \cite{NP}.
Also, in contrast to the real case, the complex cross cap points are not boundary points of the set of double values.
Third, one could try to relate the above results to the immersion on the boundary in our case. Then, however, the dimensions do not match and these immersions do not have triple points or singular points whatsoever. What is more, the Smale invariant formula \eqref{eq:eszsmale} hints that $t(f)$ and $l(f)$ should really be considered for the membranes and not the boundary.

\subsection{Smale invariant formulas}

If $M^4$ is an oriented 4-manifold with boundary, the `defect' of the equation \eqref{eq:intek4} provides information about the restriction of the map to the boundary. In the simplest case, the manifold $M^4$, with boundary $\partial M^4 $ diffeomorphic with $S^3$, is mapped to the upper half space $\mathbb{R}^6_+ = \{(x_1, \dots, x_6) \in \mathbb{R}^6 \ | \ x_6 > 0 \}$ with a generic map
\begin{equation}
\hat{f}: M^4 \to \mathbb{R}^6_+,
\end{equation}
whose restriction is assumed to be a stable immersion
\begin{equation}
f:=\hat{f}|_{\partial M^4}: S^3 \looparrowright \mathbb{R}^5.
\end{equation}
In this case $\hat{f}$ is referred as a singular Seifert surface of the immersion $f$.

Recall that the immersions of $S^3$ to $\mathbb{R}^5$ are classified up to regular homotopy by an integer valued invariant called Smale invariant and denoted by $\Omega$. 
That is, two immersions $f_1, f_2: S^3 \looparrowright \mathbb{R}^5$ are regular homotopic if and only if $\Omega(f_1)=\Omega(f_2)$, and for every integer $n \in \mathbb{Z}$ there is an immersion $g: S^3 \looparrowright \mathbb{R}^5$ with Smale invariant $\Omega(g)=n$. 
The Smale invariant can be constructed in many different ways, see for example \cite{smale, hughes-melvin, NP, gtezis}. Eventually the Smale invariant $\Omega(f)$ is constructed  as an element of the homotopy group $\pi_3 (SO(5))$, which is isomorphic with the infinite cyclic group $(\mathbb{Z}, +)$.

Then by \cite{ESz} the Smale invariant $\Omega(f)$ of a stable immersion $f: S^3 \looparrowright \mathbb{R}^5$ can be expressed with the invariants of a singular Seifert surface $\hat{f}: M^4 \to \mathbb{R}^6_+$ and $L$ as
\begin{equation}\label{eq:eszsmale}
\Omega(f)=\frac{1}{2} (3 \sigma (M^4)+3 t (\hat{f})-3 l (\hat{f})+ L(f)).
\end{equation}
Several variants and generalizations of the Ekholm--Sz\H{u}cs formula~\eqref{eq:eszsmale} appeared in the literature, see \cite{hughes-melvin, saeki, juhasz, ekholm-szucs-egzotikus} or the brief summary of these results in \cite[Ch.2]{gtezis}.

For immersions $\Phi|_{\mathfrak{S}}: \mathfrak{S} \cong S^3 \looparrowright S^5$ associated with finitely determined holomorphic germs $\Phi: (\mathbb{C}^2, 0) \to (\mathbb{C}^3, 0)$ N\'{e}methi and the first author \cite{NP} proved the `holomorphic Smale invariant formula'
\begin{equation}\label{eq:npsmale}
\Omega(\Phi|_{\mathfrak{S}}) =-C(\Phi).
\end{equation}
The proof of this formula is self-contained in the sense that it is independent of the above results. 

A singular Seifert surface for $\Phi|_{\mathfrak{S}}$ can be constructed from a holomorphic stabilization $\Phi_{\lambda}$ of $\Phi$ 
by a canonical $\mathcal{C}^{\infty}$ stabilization of the complex Whitney umbrella points. In this way the Ekholm--Sz\H{u}cs formula~\eqref{eq:eszsmale} can be applied. By comparing it with the equation \eqref{eq:npsmale} and using calculations on concrete examples, \cite{NP} proves the main theorem \eqref{eq:main} of this article, namely 
\begin{equation}
L(\Phi|_{\mathfrak{S}})=C(\Phi)-3 T(\Phi).
\end{equation}
The evolution of these results is summed up by Figure \ref{fig:gondterkep}.

However, the proof of each Smale invariant formula is rather complicated, and the identification of the signs of the terms are widely nontrivial (see the next paragraph).
Furthermore the correspondence \eqref{eq:main} becomes important in the proof of the topological invariance of $C$ and $T$.
This was the motivation to publish a new direct proof for \eqref{eq:main}, which does not use any of the above results -- although the techniques are similar to those ones used in their proofs.
An additional benefit of our proof is the simple identification of the sign: \eqref{eq:main} is sign correct with the $L_1$ version of $L$. This fact has further consequences for the singular Seifert surface formula \eqref{eq:eszsmale} as explained in the next paragraph.

\subsection{Remarks on sign and orientation}

This paragraph is a brief summary of the issues related to the signs in the Smale invariant formulas. We unravel an imprecision in \cite{ESz} and \cite{saeki}:  although the linking invariant $L$ is defined in these articles using the construction denoted by $L_2$ in Subsection~\ref{ss:eszinv}, the Smale invariant formula \eqref{eq:eszsmale} is satisfied by using $L=L_1$.

The Smale invariant does not have a canonical sign, since by default it is an element of the group $\pi_3(SO(5)) \cong (\mathbb{Z}, +)$. 
To identify this group with $\mathbb{Z}$, one has to fix a generator in $\pi_3(SO(5))$ and declare it to be $+1$. 
That was done in \cite{NP}, and the formula \eqref{eq:npsmale} is sign-correct with that fixed generator. In other words it is proved that the Smale invariant of the immersion associated with the complex Whitney umbrella is $-1$ times the fixed generator.

The formula \eqref{eq:eszsmale} is proved in \cite{ESz} without considering the sign of the Smale invariant.
More precisely, they proved that the right hand side of the formula is a complete regular homotopy invariant, therefore it must agree with the Smale invariant up to sign. Nevertheless it is shown in \cite{NP} that the foruma \eqref{eq:eszsmale} is correct with the fixed generator of $\pi_3(SO(5))$. 

However the sign of $L$ is not specified directly in \cite{NP, gtezis}. It is chosen to satisfy the Ekholm--Sz\H{u}cs formula \eqref{eq:eszsmale} with this choice. For example, the invariant $L(\Phi|_{\mathfrak{S}})$ of the complex Whitney umbrella $\Phi(s, t)=(s^2, st, t)$ is computed in \cite[10.1.]{NP} up to sign by using the `$L_2$' construction (see also \cite[3.7.1]{gtezis}), resulting $L(\Phi|_{\mathfrak{S}})=\pm 1$.  Using the sign convention adapted to the formula \eqref{eq:eszsmale}, $L(\Phi|_{\mathfrak{S}})$ of the complex Whitney umbrella is declared to be $+1$, and $L=C-3T$ is concluded with this sign convention. 

Now, from the proof of Lemma~\ref{le:whl1} it is clear that $L_1(\Phi|_{\mathfrak{S}})=+1$ for the Whitney umbrella, hence $L_2(\Phi|_{\mathfrak{S}})=-1$ by part (c) of Proposition~\ref{pr:Leq}. Therefore to make the Ekholm--Sz\H{u}cs Smale invariant formula  \eqref{eq:eszsmale}  correct, $L$ has to defined to be $L_1$, in contrast to the definitions given in \cite{ESz} and \cite{saeki}. Note that by changing the sign of $L$ in the formula not only the sign of the right hand side changes, but the absolute value changes as well. 

On the other hand,  in the calculation of $L_2(\Phi|_{\mathfrak{S}})$ of the complex Whitney umbrella  in \cite[10.1.]{NP} the sign of the intersection point can be determined directly. Both membranes ($\Phi(\mathfrak{B})$ and $H$ in \cite{NP} and \cite{gtezis}) has complex (but not holomorphic) parametrization. These parametrizations induce the correct orientations in the sense that the induced orientation on the boundary agrees with the original orientation of the boundary. A direct calculation of the determinant shows that the intersection point has negative sign. Hence $L_2(\Phi|_{\mathfrak{S}})=- 1$ can be discovered directly, which is equal to $-L_1(\Phi|_{\mathfrak{S}})$ according to Proposition \ref{pr:Leq}. 


\begin{rem}
By default, the orientation induced on the boundary of an oriented manifold depends on a choice of a convention, called `boundary convention', for example `outward normal first'. 
Although, at first sight, the boundary convention seems to play a key role in the identification of the signs of the Smale invariant formulas and $L$, this is not the case.

The correct sign of the formulas \eqref{eq:eszsmale} and \eqref{eq:npsmale} are independent of the choice of the boundary convention. Briefly speaking its reason is that in the construction of the Smale invariant $S^3$ is considered as the boundary of the 4-ball in $\mathbb{R}^4$. By changing the boundary convention, the orientation of the boundary of the singular Seifert surface changes, as well as the orientation of $S^3=\partial B^4$ in the construction of the Smale invariant, but the value of the Smale invariant and the right hand side of the formulas remain the same. See \cite{NP} or \cite[Ch.3]{gtezis} for details.

The invariant $L$ of finitely determined holomorphic germs $\Phi$ is also independent of the choice of the boundary convention. Recall that $L(\Phi|_{\mathfrak{S}})$ is defined as the intersection number of two oriented `membranes' in $B^6$ whose boundaries are $\Phi(\mathfrak{S})$ and $\widetilde{\Upsilon}=\partial \widetilde{\Sigma}$ respectively.
Although $\Phi(\mathfrak{S})$ and $\widetilde{\Upsilon}=\partial \widetilde{\Sigma}$ are originally oriented as the boundaries of $\Phi (\mathfrak{B})$ and $\widetilde{\Sigma}$ after choosing a boundary convention, all in all, the correct orientations of the membranes do not depend on the choice of the boundary convention. Indeed, the correct orientation means that the membrane induces the same orientation on the boundary as the original membrane, whichever boundary convention is used. Cf. \cite{NP, gtezis}.
\end{rem}

\section*{Funding and competing interests}

\subsection*{Declarations}

GP has been supported by the National Research Development and Innovation Office (NKFIH) through the OTKA Grant FK 132146. 

AS was partially supported by the \'Elvonal (Frontier) grant KKP126683 of the NKFIH, and Central European University. 

The authors declare they have no financial interests.

\renewcommand{\refname}{References}


\begin{thebibliography}{99}
	
	\bibitem{dBM}
	de Bobadilla, J. F.; Peñafort, G.; Sampaio, E.:
	Topological invariants and Milnor fibre for A-finite germs $\C^2 \rightarrow \C^3$,
	Accepted in Da Lat University Journal of Science (2021),
	arXiv:1912.13056v2 [math.AG].
	
	\bibitem{ekholm3}
	Ekholm, T.:
	Differential 3-knots in 5-space with and without self-intersections,
	Topology, {\bf 40}(1), (2001), 157--196.
	
	\bibitem{ekholm4}
	Ekholm, T.: Invariants of generic immersions,
	Pac. Journ. Math. {\bf 199}, no. 2, (2001), 321--346.
	
	\bibitem{ESz}
	Ekholm, T; Sz{\H{u}}cs, A.:
	Geometric formulas for {S}male invariants of codimension two
	immersions,
	Topology, { \bf 42}(1), (2003), 171--196.
	
	\bibitem{trip}
	Ekholm, T; Sz{\H{u}}cs, A.:
	On the triple points of singular maps,
	Commentarii Mathematici Helvetici, {\bf 77}, (2002), 408--414.
	
	\bibitem{ekholm-szucs-egzotikus}
	Ekholm, T; Sz{\H{u}}cs, A.:
	The group of immersions of homotopy (4k-1)-spheres,
	Bulletin of the London Mathematical Society, {\bf 38}(01), (2006), 163–-176.
	
	
	\bibitem{hughes-melvin}
	Hughes, J. F.; Melvin, P. M.:
	The {S}male invariant of a knot,
	Comment. Math. Helvetici, { \bf 60}(1), (1985), 615--627.
	
	\bibitem{juhasz}
	Juh\'asz, A.: A geometric classification of immersions of 3-manifolds into 5-space,
	manuscripta mathematica, { \bf 117}, (2005), 65--83.
	
	\bibitem{mararmulti} Marar, W. L.;  Mond, D.:
	Multiple Point Schemes for Corank 1 Maps,
	J. the London Math. Soc. (2),
	{\bf  39},  (1989),  553--567.
	
	
	\bibitem{slicing} Marar, W. L.;  Nu\~{n}o--Ballesteros, J. J.:
	Slicing corank 1 map germs from $ \C^2 $ to $ \C^3 $,
	Quart. J. Math,
	{\bf  65}, (2014), 1375--1395.
	
	\bibitem{nunodouble} Marar, W. L.;  Nu\~{n}o--Ballesteros, J. J. and Pe\~{n}afort-Sanchis, G.:
	Double point curves for corank 2 map germs from $\C^2$ to $ \C^3$,
	Topology and its Applications,
	{\bf  159},  (2012),  526-536.
	
	
	\bibitem{Mond0}
	Mond, D.: On the classification of germs of maps from $ \R^2 $ to $\R^3$,
	Proc. London Math. Soc., {\bf 50}, (1985),
	333--369.
	
	\bibitem{Mond2}
	Mond, D.:
	Some remarks on the geometry and classification of germs of maps from
	surfaces to 3-space,
	Topology, {\bf 26}(3), (1987), 361--383.
	
	\bibitem{Mondvan}
	Mond, D.:
	\newblock Vanishing cycles for analytic maps,
	Singularity Theory and its Applications, Lecture Notes in Mathematics, { \bf 1462} Springer, (1991), Berlin, Heidelberg, 221-234.
	
	\bibitem{mond-ballesteros}  Mond, D.; Nu\~{n}o--Ballesteros, J. J.: Singularities of Mappings, Springer (2020)
	
	\bibitem{mondfitting}
	Mond, D.; Pellikaan, R.:
	Fitting ideals and multiple points of analytic mappings,
	Algebraic geometry and complex analysis, 
	Springer, (1989), 107--161.
	
	
	\bibitem{NP} N\'emethi, A.; Pint\'er, G.:
	Immersions associated with holomorphic germs,
	Comment. Math. Helv., {\bf 90}, (2015),  513--541.
	
	\bibitem{NP2} N\'emethi, A.; Pint\'er, G.:
	The boundary of the Milnor fibre of certain non-isolated singularities,
	Period. Math. Hung.,
	{\bf 77}(1), (2018), 34--57.
	
	
	\bibitem{gtezis}
	Pint\'er, G. :
	On certain complex surface singularities, Ph.D. thesis,
	Eötvös Loránd University, (2018),
	arXiv:1904.12778 [math.AT].
	
	\bibitem{saeki}
	Saeki, O.; Sz{\H{u}}cs, A. and Takase, M.:
	Regular homotopy classes of immersions of 3-manifolds into 5-space,
	Manuscripta Math., {\bf 108}(1), (2002), 13--32.
	
	
	
	\bibitem{smale} Smale, S.:
	The classification of immersions of spheres in {E}uclidean spaces,
	Annals of Mathematics, {\bf 69}(2), (1959), 327--344.
	
	
	\bibitem{Wall}
	Wall, C. T. C.:
	Finite determinacy of smooth map-germs, Bull. London Math. Soc., {\bf 13}, no. 6, (1981),  481--539.
	
	
\end{thebibliography}
\end{document}